\documentclass[12pt,oneside,reqno]{amsart}

\usepackage{amsthm, amssymb, amsmath, amsfonts}
\usepackage{enumerate}
\usepackage[colorlinks]{hyperref}

\newtheorem{thm}{Theorem}[section]

\newtheorem{defi}[thm]{Definition}
\newtheorem{lem}[thm]{Lemma}
\newtheorem{prop}[thm]{Proposition}
\newtheorem{corl}[thm]{Corollary}
\newtheorem{rem}[thm]{Remark}

\newtheorem{exm}{Example}
\allowdisplaybreaks
\title[ Amenable representation]{Amenable unitary representations of locally compact groupoids}
\author[K. N. Sridharan]{K. N. Sridharan}
\address{K. N. Sridharan,\newline\indent Department of Mathematics,\newline\indent Indian Institute of Technology Delhi,\newline\indent New Delhi - 110016, India.}
\email{sreedu242@gmail.com}
\author[N. S. Kumar]{N. Shravan Kumar}
\address{N. Shravan Kumar,\newline\indent Department of Mathematics,\newline\indent Indian Institute of Technology Delhi,\newline\indent New Delhi - 110016, India.}
\email{shravankumar.nageswaran@gmail.com}
\begin{document}
\begin{abstract}
    Let $G$ be a second countable locally compact groupoid equipped with a Haar system $\lambda$. In this work, we introduce and develop the notion of amenability for continuous unitary representations of $G$, formulated in terms of Hilbert bundles over the unit space $G^{0}$. We prove that $G$ is amenable if and only if its left regular representation is amenable, thereby extending Bekka’s characterisation of amenable unitary representations from groups to groupoids. We further investigate the amenability of induced representations of $G$ developed in terms of continuous unitary representation and also study the representation of properly amenable groupoids. We define a topological invariant mean associated with a representation, constructed by utilising the theory of operator-valued vector measures on the unit space $G^{0}$, to characterise amenability. Finally, we investigate the relationship between the amenability of a representation and the amenability of its restriction to isotropy subgroups with an emphasis on HLS and certain transitive groupoids.
\end{abstract}
\maketitle
\section{Introduction}
Amenability has played a foundational role in the analysis of locally compact groups, with wide-ranging applications in harmonic analysis. In a significant extension of this classical theory, Bekka \cite{Bekkaamenablity} introduced the notion of amenability for unitary representations of locally compact groups. This operator-theoretic perspective provides a unifying framework that not only recovers the classical case when applied to the left-regular representation but also allows for a broader treatment of amenability in the context of representation theory. Within this framework, a group $G$ is amenable if and only if its left-regular representation on $L^{2}(G)$ is amenable, and crucially, all unitary representations of such a group inherit this property. This characterization shifts the focus from global group properties to the behavior of representations in Hilbert spaces, allowing amenability to be studied in an operator-theoretic point of view. Moreover, several classical properties of amenable groups admit natural analogues in this setting, offering powerful new insights and generalizations. 

Motivated by this perspective, the present work develops an analogous theory of amenability for continuous unitary representations of locally compact groupoids. Our approach is based on the framework introduced in \cite{bos2011continuous}, in which unitary representations of groupoids are described in terms of Hilbert bundles over the unit space $G^{0}$. This framework differs from the classical theory developed by Renault \cite{renault2006groupoid}, where representations are formulated using measurable Hilbert bundles. There, Renault obtained a  bijective relation 
with the representations of the groupoid $C^{*}$-algebra. In the present approach, one does not require the existence of a quasi-invariant measure on $G^{0}$. Furthermore, the coefficients associated with these representations are continuous functions, which allows us to work entirely within the topological setting. We focus on representations for which the associated Hilbert bundles have non-zero fibers and admit countably generated continuous sections vanishing at infinity, forming a Hilbert $C_{0}(G^{0})$-module. This setting provides a natural generalization, allowing us to extend operator-theoretic notions of amenability to the groupoid context and explore new characterizations.

Hilbert–Schmidt and trace-class operators on Hilbert $C^{*}$-module, as developed in a recent work \cite{stern2021schatten}, play a central role in our analysis. Utilizing these operator classes, we establish several equivalent characterizations of amenability for unitary representations, drawing natural parallels with the classical theory for groups and existing results for groupoids. In particular, we prove that every continuous unitary representation of a second countable amenable groupoid is amenable, and that a groupoid is amenable if and only if its left-regular representation is amenable. Later, we discuss the amenability of the induced representation of $G$ introduced in \cite{sridharan2025induced}. We briefly touch upon the concept of properly amenable groupoids and properties of its representations.

In addition, we define a topological invariant mean associated with a representation, analogous to the notion of a topological invariant mean for groupoids. This construction is formulated in terms of operator-valued vector measures on the unit space. To facilitate this approach, we use structural properties of trace-class operators on Hilbert $C^{*}$-modules and briefly examine the corresponding dual space structure. We further investigate the relationship between the amenability of a representation and the amenability of its restriction to isotropy subgroups. In particular, we consider the HLS groupoid and establish an analogue of a result by Willet \cite[Lemma 2.4]{Willet}, relating the amenability of a representation of the HLS groupoid to that of its restriction to the parent group. This gives some important consequences on the relationship between the approximation of positive definite function, topological property(T) in the sense of \cite[Definition $3.6$]{dell2022topological} and amenability of the representation of HLS groupoid. 

The structure of the paper is as follows. In Section 2, we provide the necessary background on groupoids, amenable groupoids, Schatten class operators on countably generated Hilbert modules, and groupoid representations. Section 3 introduces the notion of amenable unitary representations and presents related results. The definition is similar to a definition of groupoid amenability given by  Renault\cite[p.$92$]{renault2006groupoid}. We show that a second-countable groupoid $G$ is amenable if and only if its left-regular representation is amenable. We also examine the connection between the amenability of the $G$- equivariant surjection $r_{G^{0}}:G/H \to G^{0}$ and the amenability of the induced representation of $G$ , $(ind_{H}^{G}(\sigma),\mu)$, for a unitary representation $\sigma$ of a closed wide subgroupoid $H$, where $\mu$ is an equivariant $r_{G^{0}}$-system. We provide a property of the representation of \textit{properly amenable} groupoid $G$. We show that a second-countable groupoid $G$ is properly amenable if and only if for every representation $\pi$, there exists a non- zero $K \in \mathcal{L}(C_{0}(G^{0},\mathcal{H}^{\pi}))$ such that its localization $K_{u}$ being a trace-class operator on $\mathcal{H}^{\pi}_{u}$, for every $u\in G^{0}$, the family $\{K_{u}\}_{u\in G^{0}}$ is equivariant  and the trace function defined by $Tr(K)(u)=Tr(K_{u})$ belongs to  $C_{b}(G^{0})$.
In Section 4, we define the concept of a topological invariant mean associated with a representation $\pi$ of $G$, and provide equivalent characterizations of amenable representations, drawing analogies with amenable groupoids. In Section 5, we show a representation-theoretic analogue of amenability of HLS groupoid and its parent group. We show that the representation of HLS groupoid associated to $(\Gamma, \{K_{n}\}_{n\in \mathbb{N}})$ is amenable if and only if its restriction to parent group $\Gamma$ is amenable. Later, we establish the existence of an amenable extension of an amenable representation of an isotropy subgroup for transitive groupoids admitting a continuous section of the range map.
\section{Preliminaries}
Let's start with some basics on groupoids.

 A groupoid is a set $G$ equipped with a partial product map $G^{2} \to G:(x,y) \to xy$, where $G^{2}\subseteq G\times G$ denotes the set of composable pairs, and an inverse map  $G \to G: x\to x^{-1}$,  satisfying the following conditions:
    \begin{enumerate}[(i)]
        \item $(x^{-1})^{-1}=x$,
        \item $(x,y),(y,z)\in G^{2}$ implies $(xy,z),(x,yz)\in G^{2}$ and $(xy)z=x(yz)$,
        \item $(x^{-1},x)\in G^{2}$ and if $(x,y)\in G^{2}$, then $x^{-1}(xy)=y$,
        \item $(x,x^{-1})\in G^{2}$ and if $(z,x)\in G^{2}$, then $(zx)x^{-1}=z.$
    \end{enumerate}

    A topological groupoid consists of a groupoid $G$ equipped with a topology that is compatible with its algebraic structure, in the sense that:
    \begin{enumerate}[(i)]
        \item the inverse map $G\to G: x\to x^{-1}$ is continuous,
    
        \item the product map $G^{2}\to G: (x,y)\to xy$ is continuous where $G^{2}$ is endowed with the subspace topology inherited from $G \times G$. 
    \end{enumerate}
     We restrict our attention to second countable locally compact, Hausdorff groupoids. The unit space $G^{0}\subseteq G$, equipped with the subspace topology, is itself a locally compact Hausdorff space. Moreover, the range and domain maps ( denoted $r,d:G \to G^{0}$ respectively) are continuous.
      A left Haar system for $G$ consists of a family  $\{\lambda^{u}: u\in G^{0}\}$ of positive radon measures on $G$ such that, 
\begin{enumerate}[(i)]
    \item the support of the measure $\lambda^{u}$ is $G^{u}$,
    \item for any $f\in C_{c}(G),u\to \lambda(f)(u):= \int f d\lambda^{u}$ is continuous and
    \item for any $x\in G$ and $f \in C_{c}(G)$,\[\int_{G^{d(x)}}f(xy)d\lambda^{d(x)}(y)=\int_{G^{r(x)}}f(y)d\lambda^{r(x)}(y).\]  
\end{enumerate}
According to \cite[Proposition $2.4$]{renault2006groupoid}, if $G$ is a locally compact groupoid with a left haar system, then the range map $r$ is an open map. More details on groupoids can be referred to \cite{renault2006groupoid,paterson2012groupoids}. When we talk about any subgroupoid, we assume that the restriction of range and domain maps are open under relative topology.

\textbf{Amenable Groupoid}:\cite[Definition $9.13, 9.14$]{williams2019tool} A second countable locally compact groupoid $G$ is amenable if the range map $r:G \to G^{0}$ is amenable, that is, there exist \textit{approximate invariant continuous mean} for $r$ which is a continuous family of probability measures $\{m_{i}\}$, $m_{i}=\{m_i^u\}_{u\in G^{0}}$, with support of $m_i^u$ is $G^{u}$,   such that $\| \gamma^{-1}\cdot m^{u}_{i}-m_{i}^{\gamma^{-1}\cdot u}\|_{1}\to 0$ uniformly on compact subsets of $G^{0}*G=\{(u,\gamma): u=r(\gamma)\}$. A locally compact groupoid with a haarsystem $\lambda$ is \textit{topologically amenable}, as defined in \cite[Definition $9.3$]{williams2019tool}, if there is a net $\{f_{i}\} \subset C_{c}(G)$ such that 
\begin{enumerate}[(i)]
    \item the function $u \to \lambda(|f_{i}|^{2})(u)=\int_{G}|f_{i}(\gamma)|^{2}d\lambda^{u}(\gamma)$ are uniformly bounded  and
    \item the functions $\gamma\to f_{i}* f^{*}_{i}(\gamma)$ converge to a constant function $1$ uniformly on compact sets in $G$, where $f^{*}(\gamma)=\overline{f(\gamma^{-1})}$.
\end{enumerate}
The proposition \cite[Proposition $2.2.13$]{MR1799683}\cite[Corollary $9.46$]{williams2019tool} shows that the two are equivalent on a locally compact groupoid.

We now define continuous representations of groupoids.

A continuous field of Hilbert spaces over $G^{0}$ consists of a family  $\{\mathcal{H}_{u}\}_{u\in G^{0}}$ of separable Hilbert spaces together with a set of vector fields $\Gamma \in \prod_{u\in G^{0}}\mathcal{H}_{u}$,  satisfying the following conditions:
\begin{enumerate}[(i)]
    \item $\Gamma$ is a complex linear subspace of $\prod_{u\in G^{0}}\mathcal{H}_u$.
    \item For every $u\in G^{0}$, the set $\xi(u)$ for $\xi \in \Gamma$ is dense in $\mathcal{H}_{u}$.
    \item For every $\xi \in \Gamma$, the function $u \to \|\xi(u)\|$ is continuous.
    \item Let $\xi \in \prod_{u\in G^{0}}\mathcal{H}_{u}$ be a vector field; if for every $u\in G^{0}$ and every $\epsilon >0$, there exists an $\xi'\in \Gamma$ such that $\|\xi(s) -\xi'(s)\| < \epsilon$ on a neighbourhood of $u$, then $\xi \in \Gamma$.
\end{enumerate}
Given such a field, we define a topology on the disjoint union $\mathcal{H}=\sqcup_{u\in G^{0}}H_{u}$, generated by the sets of the form 
\[U(V,\xi,\epsilon)=\{h\in \mathcal{E}:\|h-\xi(p(h))\|<\epsilon,\xi\in \Gamma,p(h)\in V\}.\]
 where $V$ is an open set in $G^{0}$, $\epsilon> 0$, and $p: \mathcal{H}\to G^{0}$ is the projection of the total  space $\mathcal{H}$ to base space $G^{0}$ such that fiber $p^{-1}(u) = \mathcal{H}_{u},u \in G^{0}$.Under this topology, the map $p$ is continuous, open, and surjective.

Endowed with the above topology,  $(\mathcal{H},\Gamma)$ is referred to as a continuous Hilbert bundle as defined in \cite[Section $13$, Theorem $13.18$]{fell1988representations}, and $\Gamma$ forms the space of continuous sections of the Hilbert bundle.

The space of continuous sections vanishing at infinity is denoted by $C_{0}(G^{0},\mathcal{H})$. The space $C_{0}(G^{0},\mathcal{H})$ is fiberwise dense  and forms a Banach space with respect to the supremum norm. It also forms a Hilbert module over $C_{0}(G^{0})$. 

A  continuous unitary  representation of groupoid $G$ is a double $(\mathcal{H}^{\pi},\pi)$, where $\mathcal{H}^{\pi}=\{\mathcal{H}^{\pi}_{u}\}_{u\in G^{0}}$ is a continuous Hilbert bundle over $G^{0}$ such that  
\begin{enumerate}[(i)]
    \item  $\pi(x) \in \mathcal{B}\left(\mathcal{H}^{\pi}_{d(x)}, \mathcal{H}^{\pi}_{r(x)}\right)$ is a unitary operator, for each $x \in G$,

\item $\pi(u)$ is the identity map on $\mathcal{H}^{\pi}_{u}$ for all $u \in G^{0}$,

\item $\pi(x) \pi(y)=\pi(x y)$ for all $(x, y) \in G^{2}$,

\item $\pi(x)^{-1}=\pi\left(x^{-1}\right)$ for all $x \in G$,

\item $x \to \pi_{\xi,\eta}(x)= \langle \pi(x) \xi(d(x)),\eta(r(x))\rangle$ is continuous for every $\eta,\xi \in C_{0}(G^{0},\mathcal{H}^{\pi})$.
\end{enumerate}

Suppose $(\mathcal{H}^{\pi},\pi)$ and $(\mathcal{H}^{\pi'},\pi')$ be two continuous representations of groupoid $G$, then $Mor(\pi,\pi')$ refers to the bundle morphisms $T:\mathcal{H}^{\pi} \to \mathcal{H}^{\pi'}$ such that 
\[ T_{r(x)}\pi(x)=\pi'(x)T_{d(x)},~~ x\in G.\]
Note that $Mor(\pi,\pi')$ is a vector space.

We say that two representations $\pi$ and $\pi'$ are unitarily equivalent, denoted $\pi \sim \pi'$, if there exist $T \in Mor(\pi,\pi')$ such that each $T_{u}$ is unitary operator. In this case $T$ will form a unitary  operators between the corresponding sections vanishing at infinity which form a Hilbert $C_{0}(G^{0})$-module.

More details can be referred to in \cite{bos2007groupoids,bos2011continuous}.

A function $\phi\in C(G)$ is said 
to be \emph{positive definite} if for all
$u\in G^{0}$ and all $f\in C_{c}(G)$ we have
\begin{equation*}
\iint
\phi(y^{-1}x)\,
f(y)\,
\overline{f(x)}
\,d\lambda^{u}(x)\,
d\lambda^{u}(y)
\ge 0.
\end{equation*}
It is also proved in \cite[Theorem $1$]{paterson2004fourier}, that $\phi$ is positive definite if and only if $\phi$ is a coefficient of the form
$\pi_{\xi,\xi}$ for some $G$-Hilbert bundle $\mathcal{H}^{\pi}$ and $\xi \in  C_{b}(G^{0},\mathcal{H}^{\pi}) $. A positive definite function $\phi$ is said to be associated with the representation $\pi$ if there exists a bounded continuous section $\xi \in C_{b}(G^{0},\mathcal{H}^{\pi})$ satisfying $\phi(x)=\pi_{\xi,\xi}(x).$

We briefly recall definitions and results on Schatten class operators $\mathcal{L}^{p}(\mathcal{E}_{A})$, where $\mathcal{E}_{A}$ is a countably generated right Hilbert module over a commutative $C^{*}$-algebra $A$; see \cite{stern2021schatten}. For an introduction to Hilbert $C^{*}$-modules, see \cite{lance1995hilbert}.
\begin{defi}
    The $p$-th Schatten class $\mathcal{L}^{p}(\mathcal{E}_{A})$ for $1\leq p< \infty$ is the space of all endomorphisms $T \in \mathcal{L}(\mathcal{E}_{A})$ for which the function $tr|T|^{p}:\widehat{A}\to \mathbb{R}\cup \{\infty\}, \chi \to tr|\chi_{*}T|^{p}$ lies in $A$.
    
\end{defi}

Now, we provide more results related to trace-class operators $\mathcal{L}^{1}(\mathcal{E}_{A})$ and  Hilbert-Schmidt operators $\mathcal{L}^{2}(\mathcal{E}_{A})$, which forms a right Hilbert $A$-module under the following innerproduct. 
\begin{defi}
    The pairing $\langle\cdot,\cdot\rangle_{2}:\mathcal{L}^{2}(\mathcal{E}_{A})\times \mathcal{L}^{2}(\mathcal{E}_{A})\to A$ is given by
    $$ \langle S, T\rangle_{2}= \frac{1}{4} \sum_{k\in \mathbb{Z}/4\mathbb{Z}} i^{k}tr|T+i^{k}S|^{2}.$$
\end{defi}
\begin{prop}
    For $S,T \in \mathcal{L}^{2}(\mathcal{E}_{A}) $ and a character $\chi$ of $A$ the pairing $\langle S, T\rangle_{2}$ satisfies $\chi(\langle S, T\rangle_{2})=tr((\chi_{*}S)^{*}\chi_{*}T).$ Moreover, if $\mathfrak{e}$ is a frame of $\mathcal{E}_{A}$, then the series $\sum_{i=1}^{\infty}\langle S\mathfrak{e}_{i},T\mathfrak{e}_{i}\rangle$ converges in norm to $\langle S,T\rangle_{2}$.
\end{prop}

\begin{prop}
    Let $\mathfrak{e}$ be  a frame on $\mathcal{E}_{A}$. Then, the series $\sum_{i}\langle \mathfrak{e}_{i}, T\mathfrak{e}_{i}\rangle$ converges in norm to $trT$.
\end{prop}
Many additional results on this topic can be found in \cite{stern2021schatten}. Throughout this paper, we work with left Hilbert modules, and the inner product is taken accordingly.
\section{Amenable unitary representation and amenable groupoids }

In this section, we define the notion of amenability of a continuous unitary representation and prove some of its properties. Additionally, we discuss its connection to groupoid amenability. Throughout the paper, we consider representations for which the associated Hilbert bundles have non-zero fibers and  admit countably generated continuous sections that form a left Hilbert $C_{0}(G^{0})$-module. 

Let $G$ be a locally compact groupoid and let $\pi$ be a continuous unitary representation. Suppose that the left Hilbert $C_{0}(G^{0})$-module of sections $C_{0}(G^{0},\mathcal{H}^{\pi})$ is countably generated. Denote the space of adjointable operators on this Hilbert module by $\mathcal{L}(C_{0}(G^{0},\mathcal{H}^{\pi}))$. For $T \in \mathcal{L}(C_{0}(G^{0},\mathcal{H}^{\pi}))$ and $u \in G^{0}$, we denote by $T_{u}$ the corresponding operator on the fiber $\mathcal{H}^{\pi}_{u}$ by localisation of $T$.

Following \cite[Definition 3.2]{stern2021schatten}, we let $\mathcal{L}^{2}(C_{0}(G^{0},\mathcal{H}^{\pi}))$ denote the corresponding Hilbert-Schmidt operators. $\mathcal{L}^{2}(C_{0}(G^{0},\mathcal{H}^{\pi}))$ gives rise to a continuous field of Hilbert-Schmidt operators, denoted $HS(\mathcal{H}^{\pi})$, whose fibers are $HS(\mathcal{H}^{\pi}_{u})$, the space of Hilbert-Schmidt operators on each fiber Hilbert space $\mathcal{H}^{\pi}_{u}$ for $ u\in G^{0}$. The space $\mathcal{L}^{2}(C_{0}(G^{0},\mathcal{H}^{\pi}))$ serves as an analogue of $C_{0}(G^{0},L^{2}(G,\lambda))$ in the groupoid setting.  Since the finite rank operators on a Hilbert space is dense in Hilbert Schmidt class of operators, we can easily see that the module generated by finite rank adjointable operators on $C_{0}(G^{0},\mathcal{H}^{\pi})$ is dense in  $\mathcal{L}^{2}(C_{0}(G^{0},\mathcal{H}^{\pi}))$ using \cite[Theorem $4.6$]{lazar2018selection}.

The left regular representation of $G$ on $L^{2}(G,\lambda)=\{L^{2}(G^{u},\lambda^{u})\}_{u\in G^{0}}$  by left translation is replaced in this setting by a representation, denoted  $(HS(\mathcal{H}^{\pi}),L)$, given by:
$$ L(x) T_{d(x)}=\pi(x)T_{d(x)}\pi(x^{-1}),~~ x\in G,~~T_{d(x)}\in HS(\mathcal{H}^{\pi}_{d(x)}).$$
For convenience we sometimes denote $L(x)T_{d(x)}=x\cdot T_{d(x)}$.

For any $S,T \in \mathcal{L}^{2}(C_{0}(G^{0},\mathcal{H}^{\pi})) $, the map 
 $x\to \langle \pi(x)S_{d(x)}\pi(x^{-1}),T_{r(x)}\rangle$
 is easily seen to be continuous as it is continuous with finite rank adjointable operators on $C_{0}(G^{0},\mathcal{H}^{\pi})$.

The space  of trace-class operators on $C_{0}(G^{0},\mathcal{H}^{\pi})$, denoted $\mathcal{L}^{1}(C_{0}(G^{0},\mathcal{H}^{\pi})) $, form a Finsler module over $C_{0}(G^{0})$ as defined in \cite{phillips1998modules}. With this, one can construct a bundle of trace-class operators $TC(\mathcal{H}^{\pi})$ \cite[Theorem $8(b)$]{phillips1998modules}, with fibers $TC(\mathcal{H}^{\pi}_{u})$ being the space of trace-class operators on $\mathcal{H}^{\pi}_{u}$ for $u\in G^{0}$. Here also, the map
 $$G \to TC(\mathcal{H}^{\pi}),~ x \to \pi(x)S_{d(x)}\pi(x^{-1})$$ 
 is continuous.

Now we define the amenability of a continuous unitary representation. It is analogous to a definition of groupoid amenability given by    Renault\cite[p.$92$]{renault2006groupoid}. It is also called topologically amenable in \cite[Definition $9.3$]{williams2019tool}.
\begin{defi}
  A continuous unitary representation $\pi$ is called amenable if there exists a net $\{T^{i}\}_{i\in I}$ of Hilbert-Schmidt operators on $C_{0}(G^{0},\mathcal{H}^{\pi})$ such that 
  \begin{enumerate}[i)]
      \item  $\|T^{i}\|_{2}\leq 1$ and $\|T^{i}_{u}\|_{2}\to 1$ uniformly on every compact subsets of $G^{0}$.
      \item $\langle x\cdot T^{i}_{d(x)},T^{i}_{r(x)}\rangle \to 1$ uniformly on compact subsets of $G$.
  \end{enumerate} 
\end{defi}
Note that we can replace the net with a sequence whenever the groupoid is $\sigma$- compact. 

The following results can be easily verified. We refer \cite[Definition $15$]{bos2011continuous} for the definition of conjugate representation $\overline{\pi}$ of $\pi$.
\begin{prop}
    Let $G$ be a locally compact groupoid.
    \begin{enumerate}[(i)]
        \item A continuous unitary representation of $G$ which is unitarily equivalent to an amenable representation is amenable.
      
      \item If $\pi$ is amenable, then so is its conjugate representation $\bar{\pi}$.
    \end{enumerate}
\end{prop}

Next, we will show the relationship between the notion of amenability of a groupoid with the amenability of its unitary representations.

Given $T \in \mathcal{L}^{1}(C_{0}(G^{0},\mathcal{H}^{\pi}))$, and $f \in C_{c}(G)$, we can define a weak integral
\begin{equation}
    (f\cdot T)_{u} = \int_{G} f(y) \pi(y)T_{d(y)}\pi(y^{-1})d\lambda^{u}(y), ~~u\in G^{0}.
\end{equation}  
It is easily verifiable that $f\cdot T \in \mathcal{L}^{1}(C_{0}(G^{0},\mathcal{H}^{\pi}))$. Indeed, each fiber $(f\cdot T)_{u}$ can be expressed as a finite linear combination of operators of the form $(f_{i}\cdot T_{j})_{u}$, where $f_{i}\in C_{c}(G)$ with $f_{i}\geq 0$ and $T_{j}\in \mathcal{L}^{1}(C_{0}(G^{0},\mathcal{H}^{\pi}))$ with $T_{j}\geq 0$. By the continuity of the Haar system, each $f_{i}\cdot T_{j}$ belongs to $\mathcal{L}^{1}(C_{0}(G^{0},\mathcal{H}^{\pi}))$. Thus, it follows that $f\cdot T \in \mathcal{L}^{1}(C_{0}(G^{0},\mathcal{H}^{\pi}))$. Also with this integral $\mathcal{L}^{1}(C_{0}(G^{0},\mathcal{H}^{\pi}))$ form a Banach $C_{0}(G^{0},L^{1}(G,\lambda))$-module.
\begin{lem}
    Let $(\mathcal{H}^{\pi},\pi)$ be a continuous unitary representation of groupoid $G$, with sections $C_{0}(G^{0},\mathcal{H}^{\pi})$. Then  $\pi$ is amenable if and only if there exists a net of positive trace class operators $\{S^{i}\}_{i\in I}, \|S^{i}\|_{1}\leq 1$, such that $\|S^{i}_{u}\|_{1} \to 1$ uniformly on compact subsets of  $G^{0}$ and $\|x\cdot S^{i}_{d(x)}-S^{i}_{r(x)}\|_{1} \to 0$ uniformly on compact subsets of $G$.
\end{lem}
\begin{proof}
    Suppose $\pi$ is amenable and let $\{T^{i}\}$ be the net of Hilbert-Schmidt class operators as in the definition. Let $K$ be a compact subset of $G$ and $\varepsilon >0$. There exists a Hilbert-Schmidt operator $T^{j}$ such that 
    $$|\langle x\cdot T_{d(x)}^{j},T_{r(x)}^{j}\rangle -1|< \varepsilon^{4}/32$$ and $$||T_{u}^{j}||_{2}> 1-\varepsilon^{4}/32$$ for every $x\in K$ and $u \in C$, where $C$ is a compact subset of $G^{0}$ containing $r(K)$ and $d(K)$. Hence,
    $$ \left\|x\cdot T_{d(x)}^{j}-T^{j}_{r(x)}\right\|_{2}^{2}\leq \varepsilon^{4}/8,~~~x\in K.$$
    Using \cite[Lemma $4.2$]{Bekkaamenablity}, we can see that,
    $$ \left\|x\cdot |T_{d(x)}^{j}|-|T^{j}_{r(x)}|\right\|_{2}^{2}\leq \varepsilon^{2}/4, ~~x\in K. $$

    Defining $S^{j}= |T^{j}|^{2}$ and again using \cite[Lemma $4.2$]{Bekkaamenablity}, we get the required trace class operator with $\|x\cdot S^{j}_{d(x)}-S^{j}_{r(x)}\|_{1} <\varepsilon,~~x\in K$.

     For converse, define $T^{j}= (S^{j})^{\frac{1}{2}}$ and using the inequality $(ii)$ of \cite[Lemma $4.2$]{Bekkaamenablity}, and Cauchy-Schwartz inequality, we get the required Hilbert-Schmidt operators.
\end{proof}
If $\pi$ is a representation of $G$ and $H$ be a closed subgroupoid of $G$, then $\pi_{|_{H}}$ is a representation of $H$ whose corresponding Hilbert bundle over unit space $H^{0}$ is the reduction of $\mathcal{H}^{\pi}$ to $H^{0}$, denoted $\mathcal{H}^{\pi}_{|_{H^{0}}}$. More details can be referred to in \cite[Section $13.7$]{fell1988representations}.
\begin{prop}
 If $\pi$ is an amenable representation of $G$  and $H$ a closed subgroupoid of $G$, then the restriction $\pi_{|_{H}}$ is an amenable representation of $H$.
\end{prop}
\begin{proof}
    Since $\pi$ is amenable, there exist a net of positive trace class operators $\{S^{i}\}$ as in Lemma $3.3$. One can easily see that, by \cite[Corollary $2.10$] {lazar2018selection},   the restriction of elements of $C_{0}(G^{0},\mathcal{H}^{\pi})$ to $H^{0}$ is dense in the set of continuous sections vanishing at infinity of $\mathcal{H}^{\pi}_{|_{H^{0}}}$.  For every $\xi \in C_{0}(G^{0},\mathcal{H}^{\pi}) $,  denote $\tilde{\xi} =\xi_{|_{H^{0}}}$.  Define $\tilde S^{i}\tilde{\xi}=(S^{i}\xi)_{|_{H^{0}}}$.  Therefore, $\{\tilde S^{i}\}$ is the required net of trace class operators. 
\end{proof}
\begin{corl}
   If $\pi$ is an amenable representation of $G$, then restriction of $\pi$ to isotropy subgroups is amenable. 
\end{corl}
We will discuss the converse in specific groupoids in the last section.

The following result characterizes the groupoid amenability with the amenability of its continuous unitary representations.
\begin{thm}
    For a second-countable locally compact groupoid $G$, the following conditions are equivalent
    \begin{enumerate}[(i)]
        \item $G$ is amenable.
        \item Every continuous unitary representation of $G$  is amenable. 
        \item The left regular representation $\lambda_{G}$ of $G$ is amenable.
    \end{enumerate}
    \end{thm}
    \begin{proof}
        To show $(i)\Rightarrow (ii)$,  let $\pi$ be a unitary representation, $ C_{0}(G^{0},\mathcal{H}^{\pi})$ be the continuous sections which will form a Hilbert $C_{0}(G^{0})$-module. Let $K$ be a compact subset of $G$ and $\epsilon >0$. Since $G$ is amenable by \cite[Proposition $2.2.13$]{MR1799683}, there exist $f\in C_{c}(G)^{+}$ with $\int_{G}fd\lambda^{u}=1$ for $u\in L$ with $K\subset G^{L}_{L}$ and 
     $$\int_{G}\left|f(x^{-1}y)-f(y)\right|d\lambda^{r(x)}(y)<\epsilon,~~x\in K.$$

     Take a positive operator $S'$ in $\mathcal{L}^{1}(E_{C_{0}(G^{0})})$ with $\|S'_{u}\|_{1}=1$ for $u\in L$. Then $S=f\cdot S'$, as defined in $(1)$, is positive, $\|S_{u}\|_{1}=1$ for $u\in L$ and $\|\pi(x)S_{d(x)}\pi(x^{-1})-S_{r(x)}\|_{1}<\epsilon$ for all $x\in K$. 
     
     The implication $(ii) \Rightarrow (iii)$ is trivial.
     
        To show $(iii) \Rightarrow (i)$, by Lemma $3.3$, there exist a net $\{S^{i}\}_{i\in I}$ of positive trace class operators such that $\|S^{i}_{u}\|_{1} \to 1$ uniformly on compact subsets of  $G^{0}$ 
        and $\|x\cdot S^{i}_{d(x)}-S^{i}_{r(x)}\|_{1} \to 0$ uniformly on compact subsets of $G$. 
        
        For $\phi \in C_{0}(G)$, $T_{\phi} \in \mathcal{L}(C_{0}(G^{0},L^{2}(G,\lambda))$, where $T_{\phi}(\xi)=\phi \xi$. Its localisation $(T_{\phi})_{u}$ is the multiplication operator on each fiber $L^{2}(G^{u},\lambda^{u})$.

        Define a positive finite radon measure $m_{i}^{u}(\phi)= Tr (S^{i}_{u}(T_{\phi})_{u}),~~u\in G^{0}$. Hence, $m_{i}=\{m_{i}^{u}\}_{u\in G^{0}}$ form a continuous $r_{G^{0}}$-system. In addition it can be easily verified that $m_{i}$ is approximately invariant over compact subsets of $G$.  Since $G$ and $G^{0}$ are second countable, they are exhaustible by compact sets, and hence we can find a topological approximate invariant mean  $\{m_{i}\}_{i\in \mathbb{N}}$ as defined in \cite[Definition $2.6$]{renault2015topological}. Thus, we have the required result.
    \end{proof}

 Let $H$ be a closed wide subgroupoid of $G$ and $\mu$ be a equivariant $r_{G^{0}}$-system on the $G$-space $G/H$, as defined in \cite[Section $3.2$]{williams2019tool}. The space $G/H$ is  the set of equivalence class obtained from the equivalence relation $\sim$ as follows: $l \sim g \Leftrightarrow r(l)=r(g)~ \text{and}~ l^{-1}g\in H$ . The space $G/H$ inherits a natural left $G$-space structure with  action $g\cdot xH=(gx)H$.
 Endowed with the quotient topology and quotient map $q_{H}$, it becomes a locally compact Hausdorff space. The associated moment map $r_{G^{0}}: G/H \to G^{0}, r_{G^{0}}(gH)=r(g)$ is both open and continuous.
 
 The map $r_{G^{0}}:G/H\to G^{0}, r_{G^{0}}(xH)=r(x)$, is amenable (see \cite[Definition  $2.2.2$]{MR1799683}) if it admits an approximate invariant continuous mean \cite[Definition  $2.2.1$]{MR1799683}. We say $G/H$ admits a topological invariant approximate density if there exists a family of continuous functions with the properties given in \cite[Proposition $9.42 (c)$]{williams2019tool}.

 Before moving to the next result, let us give a brief introduction on the induced representation defined in \cite{sridharan2025induced}. 

  Suppose $(\mathcal{H}^{\sigma},\sigma)$ is a continuous unitary representation of a closed wide subgroupoid $H$ with Haar system $\lambda_{H}$.  Let $C(G,\mathcal {H}^{\sigma})$ denotes the set of  continuous function $f$ such that $f(x) \in \mathcal{H}^{\sigma}_{d(x)}$ and its subspace $C_{c}(G,\mathcal {H}^{\sigma})$ be the functions with compact support. 
Define,
\begin{align*}
  \mathcal{F}_{0}^{\sigma}(G)=\Bigg\{f\in C(G,\mathcal {H}^{\sigma}):~ & q_{H}(\operatorname{supp}(f))~ \text{is compact}~ \text{and} \\&f(xh)=\sigma(h^{-1})f(x),\text{for}~ (x,h)\in G^{2},h \in H \Bigg \}. 
\end{align*}

It is proved in \cite[Proposition $3.2$]{sridharan2025induced} that  every element of $\mathcal{F}_{0}^{\sigma}(G)$ is of the form 

    \begin{equation}
        f_{\alpha}(x)= \int_{H} \sigma(h)\alpha(xh)d\lambda_{H}^{d(x)}(h),
    \end{equation}
 for some $\alpha \in C_{c}(G,\mathcal{H}^{\sigma})$.
 $\mathcal{F}_{0}^{\sigma}(G)$ form a left pre-Hilbert $C_{0}(G^{0})$-module with innerproduct
    $$ \langle f,g\rangle(u)=\int_{G/H}\langle f(x),g(x)\rangle_{\mathcal{H}^{\sigma}_{d(x)}}d\mu^{u}(xH).$$
 Its completion $\mathcal{F}^{\sigma}(G)$ or $\mathcal{F}^{\sigma}(G,\mu)$ under the norm 
 
 $$ \|f\|= \sup_{u\in G^{0}}\sqrt{\langle f,f\rangle(u)},$$
  forms a left Hilbert $C_{0}(G^{0})$-module(see \cite[Lemma $3.3$]{sridharan2025induced}).

 The Hilbert module $\mathcal{F}^{\sigma}(G)$ can be identified with a continuous section  $\Delta=\{t_{F} \in C_{0}(G^{0},\mathcal{H}^{\operatorname{ind}_{H}^{G}(\sigma)}):t_{F}(u)= F 
 + \mathcal{N}^{u},F\in \mathcal{F}^{\sigma}(G)\}$, where $ \mathcal{N}^{u}= \{f\in \mathcal{F}^{\sigma}(G): \|f\|^{2}(u)=\langle f,f\rangle(u)=0\}$, of a continuous Hilbert bundle $\mathcal{H}^{\operatorname{ind}_{H}^{G}(\sigma)}$ whose fibers $\{\mathcal{H}^{\operatorname{ind}_{H}^{G}(\sigma)}_{u}\}_{u\in G^{0}}$ can be identified with the completion of $\{f_{|_{G^{u}}}: f\in \mathcal{F}_{0}^{\sigma}(G)\}_{u\in G^{0}}$ under the norm $\sqrt{\langle f,f\rangle(u)}$.

The representation $(\mathcal{H}^{\operatorname{ind}_{H}^{G}(\sigma)},\Delta,L)$ of $G$ where $$ L(x): \mathcal{H}^{\operatorname{ind}_{H}^{G}(\sigma)}_{d(x)}\to \mathcal{H}^{\operatorname{ind}_{H}^{G}(\sigma)}_{r(x)}, ~~(L(x)f)(y)= f(x^{-1}y), $$ is defined as  representation  induced by $\sigma$ and it is denoted by $\left(ind^{G}_{H}(\pi),\mu \right)$ (see \cite[Proposition $3.5$]{sridharan2025induced}).

 Let $\mathcal{F}^{\sigma}(G)$ be the Hilbert module corresponding to the induced representation of $G$ from the representation $\sigma$ of $H$ as defined in \cite[Section $3$, Page $7$]{sridharan2025induced}. If the section $C_{0}(G^{0},\mathcal{H}^{\sigma})$ is countably generated, then we can easily see that $\mathcal{F}^{\sigma}(G)$ is also countably generated due to \cite[Lemma $3.7$]{sridharan2025induced}.  Next, we provide some results regarding the amenability of induced representation. 

 \begin{thm} Let $r_{G^{0}}:G/H\to G^{0}$ be the continuous $G$-equivariant  open surjection. 
     \begin{enumerate}[(i)]
        
         \item  If $(ind_{H}^{G}(\sigma),\mu)$ is amenable for a unitary representation $\sigma$ of $H$, then $r_{G^{0}}$ is amenable.
          \item If $G/H$ admits a topological invariant approximate density, then induced representation $(ind_{H}^{G}(1_{H}),\mu)$ is amenable.
     \end{enumerate}
 \end{thm}
 \begin{proof}
     (i) For $\phi \in C_{c}(G/H)$, the map $T_{\phi}\in \mathcal{L}(\mathcal{F}^{\sigma}(G))$, where $T_{\phi}(\xi)(x)=\phi(q_{H}(x)) \xi(x)$.
     
     Since, $(ind_{H}^{G}(\sigma),\mu)$ is amenable, there exist a net of trace class operators $\{S^{i}\}_{i\in I}$ with each having compactly supported trace, such that $\|S^{i}_{u}\| \to 1$ uniformly on every compact subset of $G^{0}$ and $\|x\cdot S^{i}_{d(x)}-S^{i}_{r(x)}\|_{1} \to 0$ uniformly on compact subsets of $G$.

     Define a positive finite radon measure $m_{i}^{u}(\phi)=Tr(S^{i}_{u}(T_{\phi})_{u}),u\in G^{0},\phi \in C_{c}(G/H)$. Thus, $m_{i}=\{m_{i}^{u}\}_{u\in G^{0}}$ form a weak approximately invariant continuous mean as defined in \cite[Definition $9.37$]{williams2019tool}. Hence, by \cite[Proposition $9.38$]{williams2019tool}, $r_{G^{0}}$ is amenable.

     $(ii)$ Suppose there exist a net of positive $C_{c}(G/H)$  functions $\{f_{i}\}_{i\in I}$ such that $\|\pi(x)f_{i}-f_{i}\|_{1}$ tends to zero uniformly on compact subsets of $G$, where $\pi = (ind_{H}^{G}(1_{H}),\mu)$,  $\int_{G/H}f_{i} d\mu^{u}\leq 1$ for all $i$ and $\int_{G/H}f_{i} d\mu^{u}$  tends to $1$ uniformly on compact subsets of $G^{0}$. We can show that $S^{i}=\langle \cdot,\xi_{i}\rangle \xi_{i},~  \xi_{i}= f_{i}^{\frac{1}{2}}$ is the required net of Hilbert-Schmidt operators that satisfies Definition $3.1$. Here we use the the following inequality. For a compact subset $K$ of $G$, suppose    $\int_{G/H}f_{i} d\mu^{u}> 1-\varepsilon$ for $u\in r(K)$ and $\|\pi(x)f_{i}-f_{i}\|_{1}< \frac{\varepsilon^{2}}{4}$
      for $x\in K$, then
\begin{align*}
       \left |\langle \operatorname{ind}_{H}^{G}(1_{H})(x)\xi_{i},\xi_{i} \rangle-1\right| &\leq \left|\int_{G/H}(f_{i}(x^{-1}y)^{1/2}-f_{i}(y)^{1/2})f_{i}(y)^{1/2} d\mu^{r(x)}(yH)\right|+ \frac{\varepsilon}{2}\\
        &\leq \left(\int_{G/H}\left|(f_{i}(x^{-1}y)^{1/2}-f_{i}(y)^{1/2})\right|^{2} d\mu^{r(x)}(yH)\right)^{1/2}+ \frac{\varepsilon}{2}\\
        &\leq \left(\int_{G/H}\left|(f_{i}(x^{-1}y)-f_{i}(y)\right| d\mu^{r(x)}(yH)\right)^{1/2}+ \frac{\varepsilon}{2}\\
        & < \varepsilon.
    \end{align*}
 \end{proof}
\begin{prop}
    Let $\pi$ be a unitary representation of $G$ and and $\rho$ be a subrepresentation of $G$ such that the $C_{0}(G^{0},\mathcal{H}^{\rho})$ is a complemented submodule of $C_{0}(G^{0},\mathcal{H}^{\pi})$. If $\rho$ is amenable, then $\pi$ is also amenable.
\end{prop}
 \begin{proof}
    Let $P$ be the projection in $\mathcal{L}(C_{0}(G^{0},\mathcal{H}^{\pi}))$ whose range is $C_{0}(G^{0},\mathcal{H}^{\rho})$. Since  $\rho$ is amenable, there exists Hilbert Schmidt operators $T^{i}$ on  $C_{0}(G^{0},\mathcal{H}^{\rho})$ as in the definition. Then $S^{i}=i\circ T^{i}\circ P$, where $i$ is the inclusion  map of $C_{0}(G^{0},\mathcal{H}^{\rho})$ to   $C_{0}(G^{0},\mathcal{H}^{\pi})$,  is the required Hilbert-Schmidt operator .
 \end{proof}

Suppose $G$ and $H$ are groupoids with representations $\pi$ and $\rho$. The representation  $\pi \times\rho$ is  the outer tensor product of representation of groupoid $G\times H$ with continuous field of Hilbert space $\mathcal{H}^{\pi}\otimes \mathcal{H}^{\rho}$ over the unit space $G^{0}\times H^{0}$. One can refer \cite[Section $15.16$]{fell1988representations} for more details on outer tensor product of Hilbert bundles. The inner tensor product $\pi \otimes \rho$ is the representation of $G$ on the Hilbert bundle   $\mathcal{H}^{\pi} \otimes^{'} \mathcal{H}^{\rho}$ over $G^{0}$ having Hilbert spaces $\{\mathcal{H}^{\pi}_{u}\otimes \mathcal{H}^{\rho}_{u}\}_{u\in G^{0}}$ as fibers. It is defined as, $(\pi \otimes \rho)(x)(v\otimes w)= \pi(x)v \otimes \rho(x)(w) \in \mathcal{H}^{\pi}_{r(x)}\otimes \mathcal{H}^{\rho}_{r(x)} $, for $v\otimes w \in \mathcal{H}^{\pi}_{d(x)}\otimes \mathcal{H}^{\rho}_{d(x)} $. 

 Here is another characterisation of amenability of representation of $\pi$.

\begin{thm}
  Let $\pi$ be a representation of groupoid $G$.  Then $\pi$ is an amenable representation of groupoid $G$ if and only if  there exist a net of positive definite functions associated with $\pi \otimes \bar{\pi}$ that tends to $1$ uniformly on compact subsets of $G$.
\end{thm} 
\begin{proof}
The finite rank operators in the Hilbert module sense are dense on $\mathcal{L}^{2}(C_{0}(G^{0},\mathcal{H}^{\pi})$. Also we can easily see that  the representation $(\mathcal{H}^{\pi \otimes \bar{\pi}},(\pi \otimes \bar{\pi}))$ is unitarily equivalent to the representation $(HS(\mathcal{H}^{\pi}),L)$ through the mapping $\xi \otimes \bar\eta \to \langle \cdot,\eta\rangle\xi$. Let $U: C_{0}(G^{0},\mathcal{H}^{\pi \otimes \bar{\pi}}) \to \mathcal{L}^{2}(C_{0}(G^{0},\mathcal{H}^{\pi}))$ be the corresponding unitary operator intertwining with both representations. 

Suppose for a compact subset $K$ and $\varepsilon>0$, there exist a $t \in C_{0}(G^{0},\mathcal{H}^{\pi \otimes \bar{\pi}})$ of the form $t=\sum_{i=1}^{n} \xi_{i}\otimes \bar{\eta}_{i},$ where $ \xi ,\eta_{i}\in C_{0}(G^{0},\mathcal{H}^{\pi})$, such that $$|1- \langle (\pi \otimes \bar{\pi})(x)t(d(x)),t(r(x))\rangle|< \epsilon, $$ for all $x\in K$. Then by unitary equivalence, we get   
\begin{align*}
\langle (\pi \otimes \bar{\pi})(x)t(d(x)),t(r(x))\rangle &=  \langle U_{r(x)} (\pi \otimes \bar{\pi})(x)t(d(x)),U_{r(x)}t(r(x))\rangle\\ &= \langle L(x)  U_{d(x)}t(d(x)),U_{r(x)}t(r(x))\rangle \\ &= \langle   x\cdot T_{d(x)},T_{r(x)}\rangle
\end{align*}
where $T= \sum_{i=1}^{n} \langle \cdot, \eta_{i}\rangle \xi_{i}$. Hence, $|1- \langle x\cdot T_{d(x)},T_{r(x)}\rangle|< \varepsilon, x\in K$. So, we can find a net of Hilbert Schmidt operators $\{T^{i}\}_{i\in \mathbb{N}}$ of  $\|T^{i}\|_{1}\leq 1$ and satisfying Definition $3.1$. 
 The other way is direct again from the unitary equivalence of representations and density of finite rank operators.
\end{proof}

\begin{prop}
  Supppose $G$ and $H$ are groupoids with representations $\pi$ and $\rho$. Then the following holds:
  \begin{enumerate}[(i)]
      \item  If $\pi$ and $\rho$ are amenable, then  $\pi \times \rho $ is amenable 
      \item Let $G=H$. If $\pi$ and $\rho$ are amenable, then  $\pi \otimes \rho$ is amenable.
  \end{enumerate}
\end{prop}
\begin{proof}
    $(i)$.  Since $\pi$ and $\rho$ are amenable and $G$ is second countable, there exist sequences of  operators $T^{i}$ and $S^{i}$ on $\mathcal{L}^{2}(C_{0}(G^{0},\mathcal{H}^{\pi}))$ and $\mathcal{L}^{2}(C_{0}(G^{0},\mathcal{H}^{\rho}))$ respectively having the required properties. Then, $\{T^{i}\otimes S^{i}\}_{i\in \mathbb{N}}$ is the required  sequence of Hilbert-Schmidt operators making $\pi \times \rho $ amenable.

     $(ii)$ follows from $(i)$ and Proposition $3.4$, since $\pi \otimes \rho= (\pi \times \rho)_{|_{\Delta_{G}}}$, where $\Delta_{G}=\left\{(x,x)|x\in G\right\}$ can be identified as a closed subgroupoid of $G\times G$.
\end{proof}

 Now we show a property of the representation of \textit{properly amenable} groupoid (\cite[Definition $2.1.13$]{MR1799683}). It is same as  a proper groupoid(see \cite[Corollary $9.26$]{williams2019tool}). If $G$ is a group, $G$ is properly amenable if and only if it is compact. More details can also be referred to in \cite[Section $9.3$]{williams2019tool}.

 \begin{thm}
     A groupoid $G$ is properly amenable if and only if for every representation $\pi$, there exists a non-zero $K \in \mathcal{L}(C_{0}(G^{0},\mathcal{H}^{\pi}))$ such that $K_{u}$ being a trace-class operator on $\mathcal{H}^{\pi}_{u}$ for every $u\in G^{0}$, the family $\{K_{u}\}_{u\in G^{0}}$ is equivariant in the sense that $x\cdot K_{d(x)}=K_{r(x)}$ for every $x\in G$ and the trace function defined by $Tr(K)(u)=Tr(K_{u})$ belongs to  $C_{b}(G^{0})$.
 \end{thm}
 \begin{proof}
     Suppose $G$ is properly amenable, there exists a continuous system of probability measures, $m=\{m_{u}\}_{u\in G^{0}}$, which are invariant. Fix a positive element $S\in \mathcal{L}^{1}(C_{0}(G^{0},\mathcal{H}^{\pi}))$, such that $\|S\|_{1}\leq 1$ and $S_{u}$ non-zero for every $u\in G^{0}$. For each $T\in \mathcal{L}(C_{0}(G^{0},\mathcal{H}^{\pi})) $, define a continuous function $\phi_{T}$, $\phi_{T}(x)=Tr(T_{r(x)}\pi(x)S_{d(x)}\pi(x^{-1}))$.
     Define $M_{u}(T_{u})= \int_{G} \phi_{T}(x)dm_{u}(x)$. Note that $M_{u}$ is positive and  non zero on compact operators of $\mathcal{H}^{\pi}_{u}$ for every $u\in G^{0}$. 
     
     Hence, there exist non zero positive operator $K_{u}\in \mathcal{L}^{1}(\mathcal{H}^{\pi}_{u}) $ such that $M_{u}(T_{u})= Tr(T_{u}K_{u})$. Also $Tr(K_{u})= M_{u}(I)$ is continuous due to continuity of $\phi_{I}(x)$.  It is easy to verify that $ x\cdot K_{d(x)}= K_{r(x)}$ for every $x \in G$.
     
     We can define an operator $K$ on $C_{0}(G^{0},\mathcal{H}^{\pi})$ using the family of operators $\{K_{u}\}_{u\in G^{0}}$. Now, we need to show $K \in  \mathcal{L}(C_{0}(G^{0},\mathcal{H}^{\pi}))$. For that, we prove  $K\xi \in C_{0}(G^{0},\mathcal{H}^{\pi})$, where $(K\xi)_{u}=K_{u}\xi_{u}$.

     Let $\xi,\eta \in C_{0}(G^{0},\mathcal{H}^{\pi}) $, then $\langle K\xi,\eta\rangle(u)=Tr(K_{u}\langle \cdot, \eta_{u}\rangle \xi_{u}) \in C_{0}(G^{0})$. Let $\mathfrak{e}=\{\mathfrak{e}_{i}\}_{i\in \mathbb{N}}$ be a frame of $ C_{0}(G^{0},\mathcal{H}^{\pi})$. The operators $T^{i}=\langle \cdot, \mathfrak{e}_{i}\rangle \xi $ is contained in the compact operators $\mathcal{K}(C_{0}(G^{0},\mathcal{H}^{\pi}))$. Let $F$ be a finite subset in $\mathbb{N}$ and $L$ be a compact neighbourhood of $u_{0} \in G^{0}$. Then for every $u \in L$,
    \begin{align*}
        \sum_{i\in F}|\langle K \xi, \mathfrak{e}_{i}\rangle|^{2}(u)= \sum_{i\in F}|Tr(T^{i}_{u}K_{u})|^{2} &= \sum_{i\in F}\Big|\int_{G}\langle x \cdot S_{d(x)} \xi,\mathfrak{e}_{i}\rangle dm_{u}(x)\Big|^{2}\\ 
        &\leq \int_{G} \sum_{i\in F}|\langle x \cdot S_{d(x)} \xi,\mathfrak{e}_{i}\rangle|^{2}dm_{u}(x).
    \end{align*}  
    Using \cite[Lemma $2.2.3$]{MR1799683}, for $\epsilon >0$, there exist a compact subset $C$ of $G$ such that $m_{u}(G\backslash C)< \epsilon/2, u\in L$. Also note that the series $ \sum_{i\in \mathbb{N}}|\langle x \cdot S_{d(x)} \xi,\mathfrak{e}_{i}\rangle|^{2}$ converges uniformly to continuous function $\langle x \cdot S_{d(x)}\xi, x \cdot S_{d(x)}\xi\rangle$ on $C$, by Dini's theorem. Hence, the series $ \sum_{i\in \mathbb{N}}|\langle K \xi, \mathfrak{e}_{i}\rangle|^{2}(u)$ converges uniformly to $\langle K\xi,K\xi\rangle$ on $L$, thus making it continuous. Also, note that $K$ is adjointable. Hence,  we conclude that $K \in \mathcal{L}(C_{0}(G^{0},\mathcal{H}^{\pi}))$.

    For the converse, take the regular representation $\lambda_{G}$. Using arguments along the same lines as those in the proof of Theorem $3.6 \Big((iii)\Rightarrow (i)\Big)$, we get a continuous family of proper finite Radon measures $m=\{m_{u}\}_{u\in G^{0}}$ which are invariant. By normalisation, we get an invariant continuous system of probability measures.
 \end{proof}
 \begin{rem}
     It can be easily observed that every representation of properly amenable groupoids are amenable.
 \end{rem}
 \begin{rem}
     Suppose that $G$ is a transitive properly amenable groupoid.  From the above discussion, there exist a family of finite-dimensional subspaces $\{\mathcal{M}_{u}\}_{u\in G^{0}}$ consisting of eigen spaces of $K_{u}$, that are $\pi$-invariant $\Big(\pi(x)\mathcal{M}_{d(x)}\subseteq \mathcal{M}_{r(x)}\Big)$. Moreover, by arguments similar to those in the proofs of  \cite[Proposition $3.18$, Lemma $4.10$]{edeko2022uniform}, one can see that $\{\mathcal{M}_{u}\}_{u\in G^{0}}$ form a $\pi$-invariant subbundle of $\mathcal{H}^{\pi}$.
 \end{rem}
\begin{corl}
    If $G$ is a compact groupoid, then for every representation $\pi$, there exists a $K \in \mathcal{L}^{1}(C_{0}(G^{0},\mathcal{H}^{\pi}))$ that is equivariant. The converse holds when $G^{0}$ is compact.
\end{corl}
 \begin{proof}
    Suppose $G^{0}$ is compact. Then a properly amenable groupoid will be a compact groupoid.  The other way is easily followed from Theorem $3.12$.
 \end{proof}

 \section{Topological Invariant Mean and Amenability}

 In this section, we define the topological invariant mean associated with a continuous unitary representation and establish its connection to the amenability of the representation. This notion is analogous to the topological invariant mean introduced in \cite{renault2015topological} in the context of groupoids. 
 
Before presenting the definitions and results, we begin by examining the dual space of $\mathcal{L}^{1}(C_{0}(G^{0},\mathcal{H}^{\pi}))$. By \cite[Theorem $3.5$, Remark $3.19$]{stern2021schatten}, this space admits an embedding into $C_{0}(G^{0}, \mathcal{L}^{1}(H))\cong C_{0}(G^{0})\otimes_{\varepsilon}\mathcal{L}^{1}(H)$, where $H$ denotes a separable Hilbert space.

 Consequently, for each element $U \in \mathcal{L}^{1}(C_{0}(G^{0},\mathcal{H}^{\pi}))^{*}$, one can invoke \cite[p. 2]{cambern1976isomorphisms} and \cite{MR2684182} and associate a vector measure $\mu \in r\sigma bv(G^{0}, B(H))$, the set of all regular
$B(H)$- valued vector measures of bounded variation on $G^{0}$ such that 
 \begin{equation}
     U(T)= \int _{G^{0}}\phi(T) d\mu 
 \end{equation} 
 where $\phi$ is the embedding of $\mathcal{L}^{1}(C_{0}(G^{0},\mathcal{H}^{\pi}))$ into $C_{0}(G^{0}, \mathcal{L}^{1}(H))$ and the integral is termed as the immediate integral of Dinculeanu \cite[p. $11$]{dinculeanu2000vector}. Also, $\|U\|=\inf \|\mu\|$, where $\|\mu\|$ is the variation norm of $\mu$, the infimum is taken over all the vector measures $\mu$ satisfying $(2)$ and this infimum is attained. We say $U\geq 0$ if $U(T)\geq 0$ for every positive $T \in \mathcal{L}^{1}(C_{0}(G^{0},\mathcal{H}^{\pi}))$. For $U \in \mathcal{L}^{1}(C_{0}(G^{0},\mathcal{H}^{\pi}))^{*}$, define $U^{*}(T)= \overline{U(T^{*})}$, then $U^{*} \in \mathcal{L}^{1}(C_{0}(G^{0},\mathcal{H}^{\pi}))^{*}$.

 Let $\eta$ be a regular complex measure on $G^{0}$, then one can define  $\mu_{\eta}\in r\sigma bv(G^{0}, B(H)) $, $\mu_{\eta}= \eta \otimes I$, where $I$ is the identity operator in $H$. It is in $C_{0}(G^{0}, \mathcal{L}^{1}(H))^{*}$. Then, the restriction of $\mu_{\eta}$ to $\mathcal{L}^{1}(C_{0}(G^{0},\mathcal{H}^{\pi}))$, denoted $\widetilde{\mu_{\eta}}$, forms a dual of $\mathcal{L}^{1}(C_{0}(G^{0},\mathcal{H}^{\pi}))$ and $\|\widetilde{\mu_{\eta}}\|\leq \|\eta\|$.

 We have already introduced the bounded operator $L(f)$ for $f \in C_{c}(G)$ on $C_{0}(G^{0},\mathcal{H}^{\pi})$, where
 $$(L(f)(T))_{u}= \int_{G} f(y) \pi(y)T_{d(y)}\pi(y^{-1})d\lambda^{u}(y).$$
For $f\in C_{c}(G), m \in \mathcal{L}^{1}(C_{0}(G^{0},\mathcal{H}^{\pi}))^{**}$, $f*m$ be the double transposition of $L(f)$ acting on $\mathcal{L}^{1}(C_{0}(G^{0},\mathcal{H}^{\pi}))^{**}$. Thus if $T \in \mathcal{L}^{1}(C_{0}(G^{0},\mathcal{H}^{\pi})) $,
$$ \langle f*m_{T}, U\rangle=U(L(f)T)= \langle m_{L(f)T}, U\rangle$$
where $T \to m_{T}$ is the standard embedding of $\mathcal{L}^{1}(C_{0}(G^{0},\mathcal{H}^{\pi}))$ to $\mathcal{L}^{1}(C_{0}(G^{0},\mathcal{H}^{\pi}))^{**}$. Let $k \in C_{0}(G^{0}), U \in \mathcal{L}^{1}(C_{0}(G^{0},\mathcal{H}^{\pi}))^{*}, m\in \mathcal{L}^{1}(C_{0}(G^{0},\mathcal{H}^{\pi}))^{**} $, we  define 
$$ \langle k\cdot U,T\rangle = U(kT)~~ \text{and}~~ \langle k\cdot m,U\rangle = m(k\cdot U).$$

 Now, we define the topological invariant mean of a continuous unitary representation $\pi$ of $G$.
 \begin{defi}
     Let $G$ be a locally compact groupoid with Haarsystem $\lambda$ and a continuous unitary representation $\pi$ on the continuous field of Hilbert space $\mathcal{H}^{\pi}$. A topological invariant mean associated with $\pi$ is an element $m$ in the closed unit ball of   $\mathcal{L}^{1}(C_{0}(G^{0},\mathcal{H}^{\pi})^{**} $ such that
     \begin{enumerate}[(i)]
      \item $m(U)\geq 0$ if $U\geq 0$,
         \item for any probability measure $\eta$ on $G^{0}$, we have $m(\widetilde{\mu_{\eta}})=1$ and
         \item for any $f \in C_{c}(G)$, we have $f*m= \lambda(f)\cdot m$ as functionals on $\mathcal{L}^{1}(C_{0}(G^{0},\mathcal{H}^{\pi}))^{*}$.
     \end{enumerate}
 \end{defi}
 \begin{lem}
     The collection of topological invariant mean is contained in the weak$^{*}$- closure of the set of positive operators in the unit ball of $\mathcal{L}^{1}(C_{0}(G^{0},\mathcal{H}^{\pi}))$, denoted $B_{+}$.
 \end{lem}
 \begin{proof}
     Suppose there exists a topological invariant mean $m$ not contained in the weak$^{*}$- closure of $B_{+}$. Then by the Hahn-Banach separation theorem, there exist $U \in  \mathcal{L}^{1}(C_{0}(G^{0},\mathcal{H}^{\pi}))^{*}$ and $r>0$ such that 
     $$ Re (U(T))<r< Re (m(U))$$
     for all $T \in B_{+}$. Hence, 
      $$  (U+U^{*})(T)<2r< m(U+U^{*})$$
for every $T \in B_{+}$.  This contradicts   \cite[Theorem A.$1$]{gluck2025increasing}, when we take the space of selfadjoint trace class operators in $\mathcal{L}^{1}(C_{0}(G^{0},\mathcal{H}^{\pi}))$, denoted $\mathcal{L}^{1}(C_{0}(G^{0},\mathcal{H}^{\pi}))_{sa}$ as a pre-ordered Banach space with the usual partial order $\leq$.
 \end{proof}

We now establish the connection between the existence of a topological invariant mean associated with $\pi$ and the amenability of $\pi$. The following theorem establishes that amenability requires only pointwise convergence of trace-class operators. The proof follows the general approach of \cite[Theorem $2.14$]{renault2015topological}, with suitable modifications.
\begin{thm}
     Let $G$ be a locally compact groupoid with Haarsystem $\lambda$ and a continuous unitary representation $\pi$ on the continuous field of Hilbert space $\mathcal{H}^{\pi}$. A topological invariant mean of $\pi$ exists if and only if there exists a sequence $\{S^{i}\}$ of positive trace class operators on $C_{0}(G^{0},\mathcal{H}^{\pi})$, such that $\|S^{i}_{u}\|_{1}\to 1$ for every $u\in G^{0}$ and $\|x\cdot S^{i}_{d(x)}-S^{i}_{r(x)}\|_{1}\to 0$ for every $x\in G$.
\end{thm}
\begin{proof}
    
Suppose there exists a sequence of positive trace class operators $\{S^{i}\}$ with the given properties. Define $m_{i}= m_{S^{i}} \in \mathcal{L}^{1}(C_{0}(G^{0},\mathcal{H}^{\pi}))^{**}  $.
Let $m$ be the weak*-cluster point of $m_{i}$. We will see that $m$ satisfies all the axioms of Definition $4.1$. 

Let $\eta$ be a probability measure on $G^{0}$, then 
$$m_{S^{i}}(\widetilde{\mu_{\eta}})= \int_{G^{0}}\phi(S^{i})d\mu_{\eta}(u)= \int_{G^{0}}Tr(S^{i}_{u}) d\eta(u) \to 1$$
by the Dominated Convergence Theorem. Hence it should hold for any subnet $\{m_{i_{j}}\}$ converging to $m$.

Let $U \in \mathcal{L}^{1}(C_{0}(G^{0},\mathcal{H}^{\pi}))^{*}$, then there exist a $\mu \in r\sigma bv(G^{0}, B(H))$  such that 
$$ \langle f*m_{i}-\lambda(f) \cdot m_{i}, U\rangle= \int_{G^{0}}\phi(L(f)S^{i}-\lambda(f) \cdot S^{i})d\mu(u).$$

\begin{align*}
    \left|\int_{G^{0}}\phi(L(f)S^{i}-\lambda(f) \cdot S^{i})d\mu(u)\right|&\leq \int_{G^{0}}\|(L(f)S^{i}-\lambda(f) \cdot S^{i})_{u}\|_{1}d|\mu|(u)\\
    & \leq \int_{G^{0}}\int_{G}|f(y)|\|y\cdot S^{i}_{d(y)}-S^{i}_{r(y)}\|_{1}d\lambda^{u}(y)d|\mu|(u)
\end{align*}
where the second inequality comes from \cite[Theorem A.$20$]{folland2016course}. Using Dominated Convergence Theorem, applied on $L^{1}(|\mu|\circ \lambda)$, we can see that,
$$\langle f*m_{i}-\lambda(f) \cdot m_{i}, U\rangle \to 0$$
The above holds for any subnet $\{m_{i_{j}}\}$ converging to $m$
 and thus for any $f \in C_{c}(G)$, we have $f*m= \lambda(f)\cdot m$. 

 Now we show the converse. Let $m$ be a topological invariant mean associated with $\pi$. Using Lemma $4.2$, there exists a net $\{m_{S^{i}}\}_{i\in I}$ where $S^{i} \in B_{+}$,  converging to $m$ in weak*-topology.
Let $\eta$ be a probability measure on $G^{0}$, then $m_{i}(\widetilde{\mu_{\eta}}) \to m(\widetilde{\mu_{\eta}})=1$.
If $\varphi \in L^{\infty}(\eta)$, then
\begin{align*}
    \left|\int_{G^{0}}Tr(S^{i}_{u}) \varphi(u) d\eta(u)-\int_{G^{0}}\varphi (u)d\eta(u)\right|&\leq \|\varphi\|_{\infty}\int_{G^{0}}(1-Tr(S^{i}_{u}))d\eta(u)\\
    &=\|\varphi\|_{\infty} (1-m_{i}(\widetilde{\mu_{\eta}})) 
\end{align*}
which tends to $0$. Since every complex measure on $G^{0}$ is of the form $\varphi\cdot \eta$ for some probability measure $\mu$, $Tr(S^{i})$ converges to $1$ in the weak topology on $C_{b}(G^{0})$ equipped with strict topology. 

Next, we can easily verify that for every $f\in C_{c}(G)$, we have $(L(f)S^{i}-\lambda(f)\cdot S^{i})\to 0$ in the weak topology on $\mathcal{L}^{1}(C_{0}(G^{0},\mathcal{H}^{\pi}))$.

Let $\mathcal{E}_{0}=C_{b}(G^{0})$ equipped with strict topology whose dual space is $M(G^{0})$ and $\mathcal{E}_{0,w}$ be the same space equipped with weak topology. For each $f\in C_{c}(G)$, let $\mathcal{E}_{f}$ be $\mathcal{L}^{1}(C_{0}(G^{0},\mathcal{H}^{\pi}))$ with its norm topology and $\mathcal{E}_{f,w}$ with weak topology. So, $\mathcal{F}= \mathcal{E}_{0} \times \prod_{f\in C_{c}(G)}\mathcal{E}_{f}$ and the same space with weak topology, $\mathcal{F}_{w}= \mathcal{E}_{0,w} \times \prod_{f\in C_{c}(G)}\mathcal{E}_{f,w}$, have the same closed convex sets. Hence, due to previous arguments $(1,(0)_{f\in C_{c}(G)})$ belongs to the closure of the convex set
$$C=\Big\{\left(Tr(S),(L(f)S-\lambda(f)\cdot S)_{f\in C_{c}(G)}\right): S\in B_{+}\Big\}.$$

So, there exists a net, again denoted $\{S^{i}\}_{i\in I}$ in $B_{+}$ such that 
\begin{equation}
    Tr(S^{i})\to 1,  ~~~~ \text {uniformly on compact subsets of}~~ G^{0}.
\end{equation} Also, $\|L(f)S^{i}-\lambda(f)\cdot S^{i}\|_{1} \to 0$, for every $f\in C_{c}(G)$, i.e.

$$\left\|\int_{G}f(y)(y\cdot S^{i}_{d(y)}-S^{i}_{r(y)})d\lambda^{u}(y)\right\|_{1} \to 0$$ uniformly on $G^{0}$.

For $F \in C_{c}(G*_{r}G)$, one can easily verify that
$$ \left\|\int_{G}F(x,y)(y\cdot S^{i}_{d(y)}-S^{i}_{r(y)})d\lambda^{r(x)}(y)\right\|_{1} \to 0$$ uniformly on $G$ due to the density of the span of functions of the form $F= f_{1}\otimes f_{2}$, where $f_{i}\in C_{c}(G)$. In particular, for a compact set $K$, the above happens for $F(x,y)=g(x)f(x^{-1}y)$ with $ g(x)\equiv 1$ on $K$ and $g,f \in C_{c}(G)$. More precisely, for $f\in C_{c}(G)$ and compact set $K$ in $G$,
\begin{equation}
  \left\|\int_{G}f(x^{-1}y)(y\cdot S^{i}_{d(y)}-S^{i}_{r(y)})d\lambda^{r(x)}(y)\right\|_{1} \to 0,~~\text {uniformly on compact subset}~~ K.
\end{equation}

In order to prove the converse, we only need to show that for $\varepsilon>0$ and compact subset $K$ of $G$, there exist a positive trace class operator $T$ such that 
$1-Tr(T_{u})< \varepsilon$ for all $u \in L$, where $L$ is a compact subset of $G^{0}$ and $K\subseteq G^{L}_{L}$ and $\|y\cdot T_{d(y)}-T_{r(y)}\|_{1}< \varepsilon$ for all $y\in K$.  Since $G$ and $G^{0}$ are second countable, they are exhaustible by compact sets, thus helping to find a sequence of $\{T^{i}\}$ with the required property.

For $\varepsilon >0$ and compact set $K$ in $G$, fix $f\in C_{c}(G)^{+}, \|f\|_{\lambda,1}\leq 1$  and $\lambda(f)\equiv 1$ on $L$. Using $(3)$ and $(4)$ there exist $S^{i}$ such that 
$$ Tr(S^{i}_{u})\geq 1-\varepsilon, ~~~ \forall u \in d(\operatorname{supp}(f))$$ and 
$$ \left\|\int_{G}f(x^{-1}y)(y\cdot S^{i}_{d(y)}-S^{i}_{r(y)})d\lambda^{r(x)}(y)\right\|_{1} < \frac{\varepsilon}{2},~~~~ \forall x\in K. $$

With this, define a trace class operator
$$T_{u}= \int_{G}f(y)\pi(y)S^{i}_{d(y)}\pi(y^{-1})d\lambda^{u}(y).$$
We can easily verify that $Tr(T_{u})>1-\varepsilon$, for all $u\in L$ and $\|x\cdot T_{d(x)}-T_{r(x)}\|_{1}< \varepsilon$ for all $x\in K$.
Thus, we get a sequence of positive trace-class operators satisfying the properties of Lemma $3.3$, hence the result follows. 
\end{proof}

Combining all those results above, we conclude with the following theorem characterizing amenability of the unitary representation $\pi$.
\begin{thm}
    Let $G$ be a locally compact second countable groupoid with a continuous unitary representation $\pi$, then the following are equivalent
    \begin{enumerate}[(i)]
        \item $\pi$ is amenable. 
        \item there exists a topological invariant mean associated with $\pi$.
        \item there exists a sequence $\{S^{i}\}$ in $B_{+}$ such that  $\|S^{i}_{u}\|_{1}\to 1$ for every $u\in G^{0}$ and $\|x\cdot S^{i}_{d(x)}-S^{i}_{r(x)}\|_{1}\to 0$ for every $x\in G$.
        \item there exists a sequence $\{S^{i}\}$ in $B_{+}$ such that $\|S^{i}_{u}\|_{1} \to 1$ uniformly on compact subsets of  $G^{0}$ and $\|x\cdot S^{i}_{d(x)}-S^{i}_{r(x)}\|_{1} \to 0$ uniformly on compact subsets of $G$.
        \item there exists a sequence $\{S^{i}\}$ in $B_{+}$ such that $\|S^{i}_{u}\|_{1} \to 1$ uniformly on compact subsets of  $G^{0}$ and $\|L(f)S^{i}- \lambda(f)\cdot S^{i}\|_{1} \to 0$ for every $f \in C_{c}(G)$.
        \item there exist a sequence of positive definite functions associated with $\pi \otimes \bar{\pi}$ that tends to $1$ uniformly on compact subsets of $G$.
    \end{enumerate}
\end{thm}
\section{Restriction of representation and Amenability}

We have already seen that if $\pi$ is an amenable representation of $G$, then its restriction to each isotropy subgroup is amenable. This naturally raises the question to what extent the converse holds. In this section, we provide examples and counterexamples illustrating this question. We also prove a lemma concerning the amenability of representations of the well-known HLS groupoid (HLS stands for Higson, Lafforgue, and Skandalis).

Let $\Gamma$ be a discrete group. An \emph{approximating sequence} for
$\Gamma$ is a sequence $(K_n)_{n\in\mathbb N}$ of subgroups of $\Gamma$
satisfying the following conditions:
\begin{enumerate}
    \item[(i)] Each $K_n$ is a normal subgroup of $\Gamma$ with finite index.
    \item[(ii)] The sequence is decreasing; that is,
    \[
        K_n \supseteq K_{n+1}, \qquad n\in\mathbb N.
    \]
    \item[(iii)] The intersection of all the subgroups in the sequence is
    trivial:
    \[
        \bigcap_{n\in\mathbb N} K_n=\{e\}.
    \]
\end{enumerate}
A pair $(\Gamma,(K_n)_{n\in\mathbb N})$, consisting of a discrete group
$\Gamma$ together with a fixed approximating sequence, is called an
\emph{approximated group}.

We now describe the groupoid construction associated with an approximated
group, following the construction of Higson, Lafforgue, and Skandalis \cite{higson2002counterexamples}.

\begin{defi}
Let $(\Gamma,(K_n)_{n\in\mathbb N})$ be an approximated group. For each
$n\in\mathbb N$, set
\[
    \Gamma_n:=\Gamma/K_n,
\]
and let
\[
    q_n:\Gamma\longrightarrow\Gamma_n
\]
denote the canonical quotient homomorphism. For notational convenience,
put
\[
    \Gamma_\infty:=\Gamma
    \qquad\text{and}\qquad
    q_\infty:=\operatorname{id}_\Gamma.
\]

As a set, define
\[
    G:=\bigsqcup_{n\in\mathbb N\cup\{\infty\}}
    \{n\}\times\Gamma_n.
\]

We equip $G$ with the topology generated by the following collection of
open sets:
\begin{enumerate}
    \item[(i)] For every $n\in\mathbb N$ and every $g\in\Gamma_n$, the
    singleton
    \[
        \{(n,g)\}
    \]

    \item[(ii)] For every $g\in\Gamma$ and $N\in\mathbb N$, the set
    \[
        U(g,N):=
        \left\{
        (n,q_n(g)):
        n\in\mathbb N\cup\{\infty\},\ n>N
        \right\}
    \]
    
\end{enumerate}

Finally, we equip $G$ with the groupoid structure induced by the group
operations on each fibre. More precisely, the unit space is
\[
    G^{(0)}
    =
    \left\{
        (n,e):n\in\mathbb N\cup\{\infty\}
    \right\},
\]
and for $(n,g)\in G$ we define
\[
    d(n,g)=r(n,g)=(n,e).
\]
The multiplication and inverse are given by
\[
    (n,g)(n,h)=(n,gh),
    \qquad
    (n,g)^{-1}=(n,g^{-1}),
\]
whenever the products are considered within the same fibre $\Gamma_n$.

\end{defi}
It is straightforward to verify that the groupoid G constructed above is a locally compact, Hausdorff, second-countable, and étale groupoid. Furthermore, its unit space $G^{0}$ can be naturally identified with the one-point compactification of $\mathbb{N}$. We refer to $G$ as the \emph{HLS groupoid associated with the approximated group} $(\Gamma,(K_{n}))$.

In \cite[Lemma $2.4$]{Willet}, Rufus Willet proved a result regarding the amenability of HLS groupoid. Here, we show analogues result in terms of continuous representation.
 \begin{lem}
     Let $G$ be an HLS groupoid associated with approximated group $(\Gamma,(K_{n}))$ and $\pi$ be a continuous unitary representation of $G$. Then $\pi$ is amenable if and only if $\tilde{\pi}$ is amenable, where $\tilde{\pi}$ is the restriction of $\pi$ to $\Gamma_{\infty}=\Gamma$.
 \end{lem}
 \begin{proof}
     Suppose $\tilde{\pi}$ is amenable, then by \cite[Theorem $5.1$]{Bekkaamenablity} there exist a sequence of normalised positive type functions associated to the restriction of representation $\pi \otimes \bar{\pi}$, denoted $\tilde{\pi} \otimes \tilde{\bar{\pi}}$, converging to $1$ on every compact subset of $\Gamma$(also see \cite[Page $48\text{-}49$]{eymard2006moyennes}).
     
     We denote the restriction of $\pi \otimes \bar{\pi}$ on $\Gamma_{n}$ by $(\pi \otimes \bar{\pi})^{n}$.  By the arguments in the proof of \cite[Proposition $4.4$]{MR4880521}, we can see that $((\pi \otimes \bar{\pi})^{n}\circ q_{n})_{\xi,\xi}(x)= \langle((\pi \otimes \bar{\pi})^{n}\circ q_{n})(x)\xi(n),\xi(n)\rangle$ converges to $\langle (\tilde{\pi} \otimes \tilde{\bar{\pi}})(x)v,v\rangle$ uniformly on finite subets of $\Gamma$ for every $v \in \mathcal{H}^{\pi \otimes \bar{\pi}}_{\infty}$ and $\xi \in C_{b}(G^{0},\mathcal{H}^{\pi \otimes \bar{\pi}})$ with $\xi(\infty)= v $.

     Let $\{K_{i}\}_{i\in \mathbb{N}}$ be an increasing sequence of finite subsets of $\Gamma$ such that $\cup_{i\in \mathbb{N}}K_{i}=\Gamma$. 

     Combining above arguments, we can see that for each $i\in \mathbb{N}$ and $K_{i}$ there exist $n_{i}\in \mathbb{N}$ and $t_{i}\in C_{b}(G^{0},\mathcal{H}^{\pi \otimes \bar{\pi}})$ with $\|t_{i}(n)\|= 1$ for all $n\geq n_{i}$ and vanishes for $n< n_{i}$, such that 
     $$ \left|1-((\pi \otimes \bar{\pi})^{n}\circ q_{n})_{t_{i},t_{i}}(x)\right| < \frac{1}{i}$$ for $n\geq n_{i}$ and for all $x\in K_{i}$.

     Thus, $\left|1- (\tilde\pi \otimes \tilde{\bar{\pi}})_{t_{i} ,t_i}(n,q_{n}(x))\right| < \frac{1}{i}$ for $n\geq n_{i}$ and for all $x\in K_{i}$.

     By discreteness of $\mathbb{N}$  and finiteness of $\Gamma_{n}$ for all $n< n_{i}$, there exist $s_{i} \in C_{b}(G^{0},\mathcal{H}^{\pi \otimes \bar{\pi}})$ with $\|s_{i}(n)\|=1$ for $n<n_{i}$ and vanishes for all $n\geq n_{i}$, such that 
     $$\left|1- (\tilde\pi \otimes \tilde{\bar{\pi}})_{s_{i} ,s_{i}}(n,q_{n}(x))\right| < \frac{1}{i}$$ for all $x \in \Gamma$, $n<n_{i}$. 
      Combining we get,
     $$\left|1- (\tilde\pi \otimes \tilde{\bar{\pi}})_{t_{i}+s_{i},t_{i}+s_{i}}(n,q_{n}(x))\right| < \frac{1}{i}$$ for all $x \in K_{i}$, $n\in \mathbb{N}$.

    Let the chosen sequence $\{n_{i}\}_{i \in \mathbb{N}}$ be an increasing sequence of natural numbers. Suppose $K$ is a compact subset of $G$ , then there exist $j\in \mathbb{N}$ and a finite subset $F$ of $\mathbb{N} \cup \{\infty\}$ such that $K \subset \cup_{n\in F}\{(n, q_{n}(K_{j}))\} $. Hence, we can find a suitable $t_{j},s_{j} \in C(G^{0},\mathcal{H}^{\pi \otimes \bar{\pi}})$ such that 
    $$\left|1- (\tilde\pi \otimes \tilde{\bar{\pi}})_{t_{j}+s_{j},t_{j}+s_{j}}(n,q_{n}(x))\right| < \frac{1}{j}$$ for all $ (n,q_{n}(x)) \in K$.

   Hence the result follows from Theorem $4.4(vi)$.
      The other way is direct from Corollary $3.5$.
 \end{proof}

One consequence of the above lemma is the following. This result is the connection between the weak containment and amenability of representations of HLS groupoid.

Let $\pi$ and $\rho$ be two continuous representations of the groupoid $G$. Then, we say $\pi$ is \textit{weakly contained} in $\rho$ denoted as $\pi \prec \rho$ if  for every  positive definite function $\phi$ associated with $\pi$, compact subset $K$ of $G$, and $\varepsilon >0$, there exist $\eta_{1},..., \eta_{n} \in C_{b}(G^{0}, \mathcal{H}^{\rho})$ such that 
$$  \left|\phi(x)- \sum_{i=1}^{n} \langle \rho(x)\eta_{i}(d(x)), \eta_{i}(r(x))\rangle\right| < \varepsilon,~~\text{for every}~~x\in K.$$

\begin{corl}
Let $\pi$ and $\rho$ be two representations of the HLS groupoid $G$ associated with the approximated group $(\Gamma,(K_n))$. Suppose that $\pi$ is amenable and $\pi \prec \rho$. Then $\rho$ is amenable.
\end{corl}
\begin{proof}
    $\pi \prec \rho$ implies the restrictions $\tilde{\pi} \prec \tilde{\rho}$. Then by \cite[Corollary $5.3$]{Bekkaamenablity} and Lemma $5.2$, we get the result.
\end{proof}
Another consequence is related to topological property(T) of HLS groupoid in the sense of \cite[Definition $3.6$]{dell2022topological} and amenability of representation of HLS groupoid.

\begin{corl}
  Let $\pi$ be a representation of HLS groupoid $G$ associated with approximated group $(\Gamma,(K_{n}))$ having topological property(T). Then $\pi$ is amenable if and only if its restriction $\tilde{\pi}$ to $\Gamma$ has finite dimensional subrepresentation.
\end{corl}
\begin{proof}
    By \cite[Proposition $4.11$ and Corollary $4.4$]{dell2022topological}, the parent group $\Gamma$ has usual Kazdan property(T). Then the proof directly follows from \cite[Corollary $5.9$]{Bekkaamenablity} and Lemma $5.2$.
\end{proof}

 Next we give a simple example where the amenability of restriction of a groupoid representation on an isotropy subgroup does not imply the amenability of the entire representation.
 \begin{exm}
     Let $G$ be the HLS groupoid associated with the approximated group $(\Gamma,(K_{n}))$, where $\Gamma= SL_{n}(\mathbb{Z}),n\geq2$. Note that $SL_{n}(\mathbb{Z}),n\geq2$ is not amenable. Let $\pi$ be the left regular representation of $G$. Then $\pi$ restricted to $\Gamma_{n}, n\in \mathbb{N}$, is amenable since $\Gamma_{n}$ is finite, but  $\pi$ is not amenable as $\pi$ restricted to $\Gamma$ is not amenable.
 \end{exm}

For a transitive groupoid, amenability of the isotropy subgroups is equivalent to amenability of the groupoid. Here, we consider the problem of extending an amenable representation of an isotropy subgroup to the entire groupoid. We consider a class of groupoids for which an amenable representation of an isotropy subgroup can be extended to an amenable representation of the groupoid. 

 Let $G$ be a transitive groupoid. Fix $u \in G^{0}$. For $v\in G^{0}$, fix an element $\gamma_{v}\in G_{u}^{v}$ and set. For every $v,w\in G^{0}$, define $\phi_{v,w}:G^{u}_{u}\to G_{v}^{w}, \phi_{v,w}(x)=\gamma_{w}x\gamma_{v}^{-1}$. It is a homeomorphism.
 
 For $u\in G^{0}$, we call a continuous function $\gamma: G^{0}\to G_{u}$ a continuous section for the range map $r:G_{u}\to G^{0}$ if $r\circ \gamma$ is identity map. In the following proposition, we employ the technique of extending a representation of an isotropy subgroup developed in \cite[Theorem 4.1]{MR4880521}. In particular, the following result holds on transitive groupoids with discrete unit space.

 \begin{prop}
      Let $G$ be a locally compact Hausdorff transitive groupoid admitting a continuous section $\gamma$ for the range map $r:G_{u} \to G^{0}$. Then for an amenable representation $\pi$ of $G^{u}_{u}$, there exist an amenable representation $\pi'$ of $G$ such that $\pi'_{|_{G^{u}_{u}}}$ is unitarily equivalent to $\pi$.
 \end{prop}
 \begin{proof}
     Given that $\pi$ is the representation of $G^{u}_{u}$ on the Hilbert space $\mathcal{K}^{\pi}$. We will define a representation $(\mathcal{H}^{\pi'},\pi')$, where the Hilbert bundle $\mathcal{H}^{\pi'}$ is a constant bundle with fibers $\mathcal{H}^{\pi'}_{v}=\mathcal{K}^{\pi},v\in G^{0}$.  For each $v\in G^{0}$, set $\gamma_{v}=\gamma(v)$. Define $\pi'(x)= \pi(\phi^{-1}_{s(x),r(x)}(x))$. So $(\mathcal{H}^{\pi'},\pi')$ is a continuous unitary representation of groupoid. As the Hilbert bundle $\mathcal{H}^{\pi'}$ is a constant bundle, the continuous section vanshing at infinity is $C_{0}(G^{0},\mathcal{H}^{\pi})$. Take an increasing sequence of compact subsets $\{K_{n}\}_{n\in \mathbb{N}}$, such that $K_{n}\subseteq \operatorname{int}(K_{n+1})$ and $\cup_{i=1}^{\infty}K_{i}=G^{0}$. Define a sequence of positive continuous functions with compact support $\{f_{n}\}_{n\in \mathbb{N}}$ on $G^{0}$ such that $f_{n}\equiv 1$ on  $r(K_{n})\cup d(K_{n})$ for every $n\in \mathbb{N}$. Also by amenability of $\pi$, there exist a sequence of positive trace-class operators $\{T'_{n}\}_{n\in \mathbb{N}}$ of unit norm on $\mathcal{H}^{\pi}$ such that $\|\pi(x)T_{i}'\pi(x^{-1})-T'_{i}\|_{1}\to 0$ for every $x\in G^{u}_{u}$. Define $T^{i}= f_{i}T'_{i}$. By Theorem $4.4(iii)$, it is easily verifiable that $T^{i}$ is the required sequence of positive trace-class operators. 
 \end{proof}

\section*{Acknowledgement}
 K. N. Sridharan is supported by the NBHM doctoral fellowship with Ref number: 0203/13(45)/2021-R\&D-II/13173.

\section*{Data Availability}
Data sharing does not apply to this article as no datasets were generated or analyzed during the current study.

\section*{competing interests}
The authors declare that they have no competing interests.
 \bibliographystyle{abbrv}
\bibliography{reff}
\end{document}